\documentclass[reqno]{amsart}
\usepackage{hyperref}
\usepackage{enumerate}
\usepackage{graphicx}
\usepackage{amsmath} 
\usepackage{amssymb}
\usepackage{amsfonts}
\usepackage{amsthm}
\usepackage{geometry}
\usepackage[table,xcdraw]{xcolor}
\usepackage{tikz}

\AtBeginDocument{{\noindent\small
Third International Conference on Applications of Mathematics to Nonlinear Sciences,\newline
\emph{Electronic Journal of Differential Equations},
Conference XX (202X), pp. xx--xx.\newline
ISSN: 1072-6691. URL: http://ejde.math.txstate.edu or http://ejde.math.unt.edu}
\thanks{\copyright 202X This work is licensed under a CC BY 4.0 license.}
\vspace{8mm}}

\begin{document} \setcounter{page}{1}
\title[\hfilneg EJDE-202X/conf/XX\hfil  RUNNING HEADER]
{On Existence of Traveling Wave of an HBV Infection Dynamics Model: A novel approach}

\author[Rupchand Sutradhar, D C Dalal \hfil EJDE-2022/conf/26\hfilneg]
{Rupchand Sutradhar, D C Dalal}

\address{Rupchand Sutradhar \newline
Indian Institute of Technology Guwahati,Guwahati, Assam, 781039, India}
\email{rsutradhar@iitg.ac.in}

\address{D C Dalal \newline
Indian Institute of Technology Guwahati,Guwahati, Assam, 781039, India}
\email{durga@iitg.ac.in}

\thanks{Published Month Day, 202X}

\begin{abstract}
 In this work,  a hepatitis B virus infection dynamics model is proposed including the spatial dependence of viruses. The existence of traveling waves for the proposed model is established through the application of the celebrated $Ger\hat{s}gorin$  theorem. The procedure followed to establish the existence of a traveling wave solution is  innovative and probably the first attempt of this particular approach.  The elasticity of  basic reproduction number with respect to some model parameters are also shown.  Furthermore,  the effects of spatial diffusivity of the viruses on infection are studied, and  it is noticed that  due to the diffusion, viruses spread rapidly throughout the liver.
\end{abstract}

\maketitle
\numberwithin{equation}{section}
\newtheorem{theorem}{Theorem}[section]
\newtheorem{proposition}[theorem]{Proposition}
\newtheorem{remark}[theorem]{Remark}
\newtheorem{definition}[theorem]{Definition}
\newtheorem{lemma}[theorem]{Lemma}
\allowdisplaybreaks

\section{Introduction} 
The hepatitis B virus (HBV) is a prominent causative agent in HBV infections worldwide. Chronic HBV infection leads to the deadly liver diseases like  cirrhosis, primary hepatocellular carcinoma (HCC), etc. Despite the availability of a highly effective vaccine for this virus, a significant number of  the population continues to grapple with the burden of this viral infection. Mainly two kinds of antiviral drugs are used to treat HBV infection: (i) viral replication inhibitors (lamivudine, adefovir, entecavir, telbivudine, and tenofovir), and (ii) immune system modulators (interferons (IFN)-alpha-2a, pegylated (PEG)-IFNalpha-2a).

In the literature, in the year 1996, Nowak et al. \cite{1996_Nowak} first proposed a HBV infection dynamics model which is commonly known as basic model. This basic model comprises three compartments, including susceptible host cells, infected cells, and free virus particles. Following the pioneering work of Nowak et al., extensive research has been carried out modifying the basic model or formulating a new model based on the biological findings available in the literature \cite{2006_Murray,2007_Ciupe,2008_Min,2015_Manna,2018_Danane_mathematical,2018_fatehi_nkcell,2021_hews_global,2023_Sutradhar_fractional}. Despite the advancements in theoretical studies and remarkable progress in medical science, HBV infection continues to pose a significant threat as  potentially fatal liver diseases. The main reasons could be the complex life cycle and replication process of the virus that is really difficult to understand. Recycling of capsids is one of the key  intracellular steps in the viral life cycle and acts a positive feedback loop \cite{2016_jun_nakabayashi}. Recently,  incorporating the recycling of capsids, Sutradhar and Dalal \cite{2023_sutradhar_recycling} proposed an improved  mathematical model on HBV infection which is  given by a system of equations \eqref{ODE_model} with slight symbolic modifications.
\begin{equation} \label{ODE_model}
	\left.
\begin{split}
	&\frac{dT}{dt}	= \lambda-\mu T-kVT,\\
	&\frac{dI}{dt}	=kVT-\delta I,\\
	&\frac{dD}{dt}=aI+\gamma(1-\alpha)D-\alpha\beta D-\delta D,\\
	&\frac{dV}{dt}=\alpha\beta D-\delta_v V,	
\end{split} 
\right\}
\end{equation}
where the biological interpretations of each model parameter and model variable are described in Table \ref{Table: Parameter describtion}.
The authors of the above model \eqref{ODE_model} assumed that cells and viruses are homogeneously distributed throughout the liver, and ignored the mobility of cells, capsids and viruses. Biological motion, characterized by the movement of living organisms, plays a pivotal role in shaping and influencing a diverse array of biological phenomena. In this study, including the random spatial mobility of viruses followed by the Fickian diffusion, we extend the model\eqref{ODE_model} and the modified reaction diffusion equation is given by the following dynamical system:
\begin{equation} \label{Eqn:PDE}
	\left.
	\begin{split}
		&\frac{\partial T}{\partial t}	= \lambda-\mu T-kVT,\\
		&\frac{ \partial I}{\partial t}	=kVT-\delta I,\\
		&\frac{\partial D}{\partial t}=aI-R_sD,\\
		&\frac{\partial V}{ \partial t}={d}_v\dfrac{\partial^2 V}{\partial x^2}+\alpha\beta D-\delta_v V,	
	\end{split}
\right\}
\end{equation} 
where $d_v$ denotes the diffusion coefficient and consider  $R_s=\alpha\beta-\gamma(1-\alpha)+\delta$. The susceptible host cells and infected cells can't move whereas HBV DNA-containing capsids can, but their movement is restricted within the single infected cells. Hence, the movements of the capsids are ignored for this time being.  Only spatial mobility of the free viruses are considered here.  The proposed model \eqref{Eqn:PDE} is visually depicted in Figure \ref*{diagrammatic representation}.
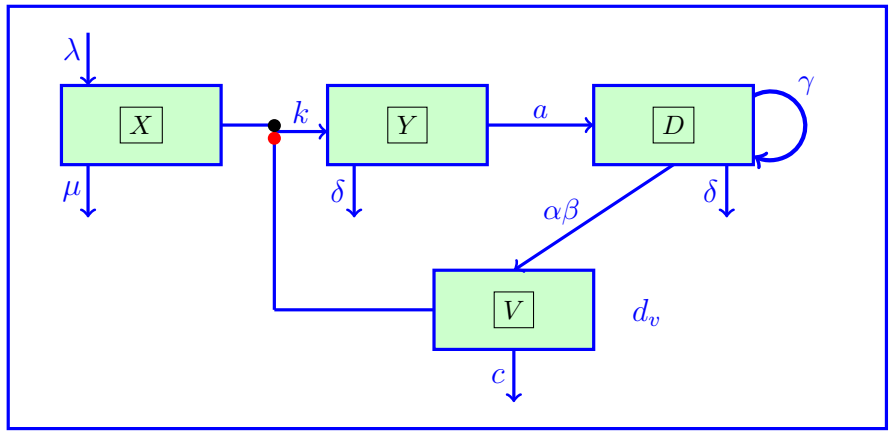
\begin{figure}[h]
	\begin{center}
	\begin{tikzpicture}[scale=0.7]
		\draw[blue, very thick] (0,-4) rectangle (16.5,4); 	 
		\draw[blue, very thick,fill = green!20!white] (1,1) rectangle (4,2.5);	 
		\draw[blue, very thick,fill = green!20!white] (6,1) rectangle (9,2.5);	 
		\draw[blue, very thick,fill = green!20!white] (11,1) rectangle (14,2.5);   
		\draw[blue, very thick,fill = green!20!white] (8,-1) rectangle (11,-2.5);	 
		\draw[blue, very thick] (4,1.75)--(5,1.75);		     
		\draw[->,blue, very thick] (5,1.627)--(6,1.627);		     
		\draw[->,blue, very thick] (9,1.75)--(11,1.75);		 
		\draw[color=blue] (10,2.0) node {\large $a$};		 
		\draw[blue, ultra thick, ->] (14,2.3) arc (-240:-478:0.65);     
		\draw[color=blue] (15,2.5) node {\large $\gamma$};		  	
		\draw[->,blue, very thick] (12.5,1)--(9.5,-1);					
		\draw[color=blue] (10.4,0.1) node {\large $\alpha\beta$};		        
		\node[draw] at (2.5,1.75) {$X$};									
		\node[draw] at (7.5,1.75) {$Y$};									
		\node[draw] at (12.5,1.75) {$D$};									
		\node[draw] at (9.5,-1.75) {$V$};									
		\draw[->,blue, very thick] (1.5,3.5)--(1.5,2.5);        		
		\draw[color=blue] (1.2,3.2) node {\Large $\lambda$};				
		\draw[->,blue, very thick] (1.5,1)--(1.5,0);					
		\draw[color=blue] (1.2,0.5) node {\Large $\mu$};				
		\draw[->,blue, very thick] (6.5,1)--(6.5,0);		        	
		\draw[color=blue] (6.2,0.5) node {\Large $\delta$};				
		\draw[->,blue, very thick] (13.5,1)--(13.5,0);		        	
		\draw[color=blue] (13.2,0.5) node {\Large $\delta$};			
		\draw[->,blue, very thick] (9.5,-2.5)--(9.5,-3.5);		        
		\draw[color=blue] (9.2,-3) node {\Large $c$};						
		\draw[color=blue] (12,-1.8) node {\Large $d_v$};	
		\draw[blue, very thick] (8,-1.75)--(5,-1.75);
		\draw[blue, very thick] (5,-1.75)--(5,1.5);	
		\fill [color=red] (5,1.5) circle (3.5pt);
		\fill [color=black] (5,1.74) circle (3.5pt);
		\draw[color=blue] (5.5,2) node {\Large $k$};					
	\end{tikzpicture}
\end{center}
\caption{The diagrammatic representation of the system \eqref{Eqn:PDE}.}
\label{diagrammatic representation}
\end{figure}

It is noteworthy that the evidences  indicate wound healing and the dissemination of solid tumors can propagate in a manner reminiscent of a traveling wave front \cite{2004_maini_travelling,2004_matzavinos_travelling}. 
Many authors have also investigated HBV infection in the context of  traveling wave and interpreted their results from various angles \cite{2010_gan_travelling,2016_duan_dynamics,2021_issa_diffusion}. All of them have  dealt with susceptible hepatocytes, infected hepatocytes and free viruses as model compartments. In this study,  the existence of  traveling wave solution of an HBV infection dynamics model is established considering the following three main factors that were not accounted in the previous studies:  
\begin{enumerate}[(i)]
	\item  HBV capsids as a separate compartment.
	\item  The recycling effects of capsids. 
	\item The diffusion of viruses.
\end{enumerate}
For the sake of mathematical simplicity, we restrict our analysis to one-dimensional space. The model \eqref{Eqn:PDE} is streamlined  by implementing the following appropriate transformations:
$$T_1=\dfrac{\mu}{\lambda}T,~I_1=\dfrac{\mu}{\lambda}I,~ D_1=\frac{\mu}{\lambda}D,~V_1=\dfrac{k}{\mu}V,~t_1=\mu t,~x'=x,~\rho_1=\frac{\delta}{\mu},~\rho_2=\frac{a}{\mu},~\rho_3=\frac{R_s}{\mu},$$
$$\rho_4=\frac{k\alpha\beta\lambda}{\mu^3},~\rho_5=\frac{\delta_v}{\mu},~\mathcal{D}_v=\mu d_v.$$
Here, all parameters $\rho_1,~\rho_2,~\rho_3,~\rho_4,~\rho_5,~\mathcal{D}_v$ are positive constants.
As a result, the following system of equations is obtained (omitting the primes on t and x for simplicity) as:
\begin{equation} \label{Eqn: PDE-simple form rho-1 to rho-5}
	\left.
	\begin{split}
		&\frac{\partial T_1}{\partial t}	= 1-T_1-V_1T_1,\\
		&\frac{ \partial I_1}{\partial t}	=V_1T_1-\rho_1 I_1,\\
		&\frac{\partial D_1}{\partial t}=\rho_2I_1-\rho_3D_1,\\
		&\frac{\partial V_1}{ \partial t}=\mathcal{D}_v\dfrac{\partial^2 V_1}{\partial x^2}+\rho_4 D_1-\rho_5 V_1.	
	\end{split}
\right\}
\end{equation}
\section{Preliminaries} \label{Preliminaries}
\begin{theorem} \cite{2012_horn_matrix} \label{2012_Horn_Johnson}
	Consider a matrix $A=[a_{ij}]\in M_n$, set of all $n\times n$ matrices. Let
	$$R_i'(A)=\sum_{j\neq i}|a_{ij}|,~i=1,2,...,n.$$
	 denote the deleted absolute row sums of the given matrix $A$.  Consider $n$ $Ger\hat{s}gorin$ discs, defined as 
	 $$\left\{ z \in \mathbb{C}: |z-a_{ii}|\leq R_i'(A)\right\}, i=1,2,...,n.$$
	Then, the  eigenvalues of the matrix $A$ lie in the union of $Ger\hat{s}gorin$ discs
	$$G(A)=\displaystyle \bigcup_{i=1}^{n}\left\{z\in\mathbb{C}: |z-a_{ii}|\leq R'_i(A)\right\}$$
	\\		
	Additionally, if, out of  $n$ discs, $k$ discs form the set $G_k(A)$ that remains disjoint from the remaining $(n-k)$ discs, then $G_k(A)$ contains exactly $k$ eigenvalues
	of $A$, counted according to their algebraic multiplicities.
\end{theorem}
\begin{definition}
	Consider a continuous dynamical system described by $\dot{X}=AX$, where $A$ is an $n\times n$ matrix and $X=(x_1,x_2,...,x_n)^T$ is an $n\times 1$ vector of dependent variables. Suppose that there are two equilibrium points of $\dot{X}=AX$ at $X^*$ and $X^{**}$. Then, a solution $\Phi (t)$ is said to be a heteroclinic orbit starting from $X^*$ to $X^{**}$ if the following conditions are satisfied.
	$$\displaystyle \lim\limits_{t \to -\infty}\Phi (t)=X^*,~\lim\limits_{t \to \infty}\Phi (t)=X^{**}$$
	or, 
		$$\lim\limits_{t \to -\infty}\Phi (t)=X^{**},~\lim\limits_{t \to \infty}\Phi (t)=X^{*}.$$

\end{definition}
\begin{table} \label{Table: Parameter describtion}
	\begin{tabular}{|c|ll|}
		\hline
		\rowcolor[HTML]{FFCCC9} 
		Variables  & \multicolumn{2}{l|}{\cellcolor[HTML]{FFCCC9}Description}                                                      \\ \hline
		\rowcolor[HTML]{96FFFB} 
		$T$          & \multicolumn{2}{l|}{\cellcolor[HTML]{96FFFB}Number of uninfected hepatocytes}                                 \\ \hline
		\rowcolor[HTML]{96FFFB} 
		$I$          & \multicolumn{2}{l|}{\cellcolor[HTML]{96FFFB}Number of infected hepatocytes}                                   \\ \hline
		\rowcolor[HTML]{96FFFB} 
		$D$          & \multicolumn{2}{l|}{\cellcolor[HTML]{96FFFB}HBV capsids}                                                      \\ \hline
		\rowcolor[HTML]{96FFFB} 
		$V$          & \multicolumn{2}{l|}{\cellcolor[HTML]{96FFFB}Viruses}                                                          \\ \hline
		\rowcolor[HTML]{FFCCC9} 
		Parameters & \multicolumn{1}{l|}{\cellcolor[HTML]{FFCCC9}Descriptions}                                            & Values \\ \hline
		\rowcolor[HTML]{96FFFB} 
		$\lambda$  & \multicolumn{1}{l|}{\cellcolor[HTML]{96FFFB}Natural growth rate of uninfected  hepatocytes}          &  $2.6\times10^7$      \\ \hline
		\rowcolor[HTML]{96FFFB} 
		$k$        & \multicolumn{1}{l|}{\cellcolor[HTML]{96FFFB}Virus to cell infection rate}                            &  $1.67\times10^{-12}$      \\ \hline
		\rowcolor[HTML]{96FFFB} 
		$\mu$      & \multicolumn{1}{l|}{\cellcolor[HTML]{96FFFB}Natural death rate of uninfected hepatocytes}            &    0.01    \\ \hline
		\rowcolor[HTML]{96FFFB} 
		$\delta$   & \multicolumn{1}{l|}{\cellcolor[HTML]{96FFFB}Natural death rate of infected hepatocytes}              &   0.053     \\ \hline
		\rowcolor[HTML]{96FFFB} 
		$a$        & \multicolumn{1}{l|}{\cellcolor[HTML]{96FFFB}Production rate of capsids from infected hepatocytes}    &    150    \\ \hline
		\rowcolor[HTML]{96FFFB} 
		$\gamma$   & \multicolumn{1}{l|}{\cellcolor[HTML]{96FFFB}Recycling rate of capsids}                               &   0.6931     \\ \hline
		\rowcolor[HTML]{96FFFB} 
		$\alpha$   & \multicolumn{1}{l|}{\cellcolor[HTML]{96FFFB}Volume fraction of capsids in favor of virus production} &    0.80    \\ \hline
		\rowcolor[HTML]{96FFFB} 
		$\beta$    & \multicolumn{1}{l|}{\cellcolor[HTML]{96FFFB}Virus production rate from infected hepatocytes}         &   0.87     \\ \hline
		\rowcolor[HTML]{96FFFB} 
		$c$        & \multicolumn{1}{l|}{\cellcolor[HTML]{96FFFB}Natural decay rate of viruses}                           &    3.8    \\ \hline
		\rowcolor[HTML]{96FFFB} 
		$d_v$        & \multicolumn{1}{l|}{\cellcolor[HTML]{96FFFB}Diffusion coefficient}                           &   0.08     \\ \hline
	\end{tabular}
\end{table}
\section{Initial and  boundary conditions}
 In order to solve the proposed model \eqref{Eqn: PDE-simple form rho-1 to rho-5},  biologically relevant initial and boundary conditions are considered. For the sake of convenience, the total length of the liver is denoted by  symbol $L$.

\subsection{Initial conditions:} 
The distributions of initial concentration of $`T,~I,~D,~V$   are as follows: 
	\begin{equation} \label{eq:initial condition}
	\left.
	\begin{split}
		&T_1(x,0)	= T_0\left\{1-\exp\left(-\frac{x^2}{\epsilon}\right)\right\},~ 0\leq x\leq L,\\
		&I_1(x,0)	= I_0\exp\left(-\frac{x^2}{\epsilon}\right),~0\leq x\leq L,\\
		&D_1(x,0)=D_0\exp\left(-\frac{x^2}{\epsilon}\right),~0\leq x\leq L,\\
		&V_1(x,0)=V_0\exp\left(-\frac{x^2}{\epsilon}\right),~0\leq x\leq L,~\epsilon=0.02,\\
	\end{split}
	\hspace{1cm}
	\right\} 
\end{equation}
where $T_0,~I_0,~D_0,$ and $V_0$ represent maximum values of $T_1,~ I_1,~D_1$ and $V_1$, respectively.
\subsection{Boundary conditions} Throughout the infection, it is assumed that there is  no flux across the liver boundary. The liver is considered as an impervious fortress, preventing any external viruses from entering or leaving. The boundary conditions are given as:
\begin{equation}\label{eq:boundary_conditions_left}
\frac{\partial V_1}{\partial x}\Big|_{x=0}=0,~~\frac{\partial V_1}{\partial x}\Big|_{x=L}=0,~t\geq 0,
\end{equation}
where $\dfrac{\partial}{\partial x}$ denotes the outward normal derivative at the boundary of the domain  $x=0$ and $x=L$.
\section{Elasticities of basic reproduction number with respect to  parameters}
The basic reproduction number of the system of equations \eqref{Eqn: PDE-simple form rho-1 to rho-5}  is given by
$R_0=\dfrac{\rho_2\rho_4}{\rho_1\rho_3\rho_5}=\dfrac{ak\lambda\alpha\beta}{R_s\delta\delta_v\mu}$.
 The equilibrium point are:
$$(1,0,0,0), ~\text{and},~
\left(\frac{\rho _1 \rho _3 \rho _5}{\rho _2 \rho _4},~ \frac{\rho _2 \rho _4-\rho _1 \rho _3 \rho _5}{\rho _1 \rho _2 \rho _4},~ \frac{\rho _2 \rho _4-\rho _1 \rho _3 \rho _5}{\rho _1 \rho _3 \rho _4},~ \frac{\rho _2 \rho _4-\rho _1 \rho _3 \rho _5}{\rho _1 \rho _3 \rho _5}\right)=\left(T_1^*,I_1^*,D_1^*,V_1^*\right)$$

\noindent	
The static quantity $R_0$ depends on nearly all the parameters of the model \eqref{Eqn:PDE}. In the prediction of evolution of HBV, the threshold number $R_0$ plays important roles. The sensitivity analysis of $R_0$ is performed here in order to determine how $R_0$ responds to the  changes in parameters.  The elasticity of a quantity $\mathcal{Q}$ with respect to the parameter $p$ is denoted by $\mathcal{E}_p^\mathcal{Q}$ \cite{2015_martcheva_introduction} and  defined as
\begin{align*}
	\mathcal{E}_p^\mathcal{Q}=\frac{p}{\mathcal{Q}} \frac{\partial \mathcal{Q}}{\partial p}=\frac{\partial~ \ln \mathcal{Q}}{\partial ~\ln p}.
\end{align*} 
\noindent
Elasticity of $\mathcal{Q}$ is positive if it increases  while the value of $p$ increases and vice-versa. Negative value of $\mathcal{E}_p^\mathcal{Q}$ means that the quantity $\mathbb{Q}$ moves in the opposite direction of $p$.
\begin{align}
	&\mbox{Elasticity of}~R_0~\mbox{w.r.t}~\alpha~\left(\mathcal{E}_\alpha^{{R}_0}\right)=\frac{\alpha}{R_0} \frac{\partial R_0}{\partial \alpha}=\frac{\delta -\gamma }{R_s },\label{elastricity_alpha}\\
	&\mbox{Elasticity of}~R_0~\mbox{w.r.t}~\beta~\left(\mathcal{E}_\beta^{{R}_0}\right)=\frac{\beta}{R_0} \frac{\partial R_0}{\partial \beta}=\frac{(\alpha -1) \gamma +\delta }{R_s},\label{elastricity_gamma}\\	
	&\mbox{Elasticity of}~R_0~\mbox{w.r.t}~\gamma~\left(\mathcal{E}_\gamma^{{R}_0}\right)=\frac{\gamma}{R_0} \frac{\partial R_0}{\partial \gamma}=\frac{\gamma(1-\alpha)}{R_s}. \label{elastricity_beta}
\end{align}
From the expressions given in the equations \eqref{elastricity_alpha},\eqref{elastricity_gamma}, and \eqref{elastricity_beta}, the followings are observed: 
\begin{enumerate}
	\item $ \displaystyle \mathcal{E}_\alpha^{{R}_0}>0$ if $\delta>\gamma$ and $\mathcal{E}_\alpha^{{R}_0}<0$ if $\delta<\gamma$. 
	So, the  nature of functions (positive or negative) of volume fraction of capsids  depends on the death rate of infected hepatocytes and recycling rate of capsids. 
	\item $\displaystyle \mathcal{E}_\beta^{{R}_0}>0$ if $\delta>(1-\alpha)\gamma$ and $\mathcal{E}_\beta^{{R}_0}<0$ if $\delta<(1-\alpha)\gamma$. As in the case of $\alpha$, similar kind of conclusion can be drawn on $\mathcal{E}_\beta^{{R}_0}$. 
	\item $\displaystyle \mathcal{E}_\gamma^{{R}_0}>0$ since $(1-\alpha)>0$ always, which implies that recycling of capsid acts as a positive feedback loop in infection.
\end{enumerate}
\noindent
Using the values of parameters shown in the Table \ref{plot_bar: elastricity}, one can find one possible relationship between $\alpha$, $\beta$, $\gamma$ and $R_0$: 	
\begin{align*}
	\mathcal{E}_\alpha^{{R}_0}\approx-1.05,~~ \mathcal{E}_\beta^{{R}_0}\approx-0.14,~~ \mbox{and} ~~ ~\mathcal{E}_\gamma^{{R}_0}\approx 0.23.
\end{align*}	
\noindent
These elasticity values mean that  1\% increase in $\alpha$, $\beta$ and $\gamma$ produce 1.05\%, 0.14\%   decrease and 0.23\% increase
in $R_0$ (refer Figure \ref{plot_bar: elastricity}). 
\begin{figure}[h!]
	\begin{center}
		\includegraphics[width=12cm,height=8cm]{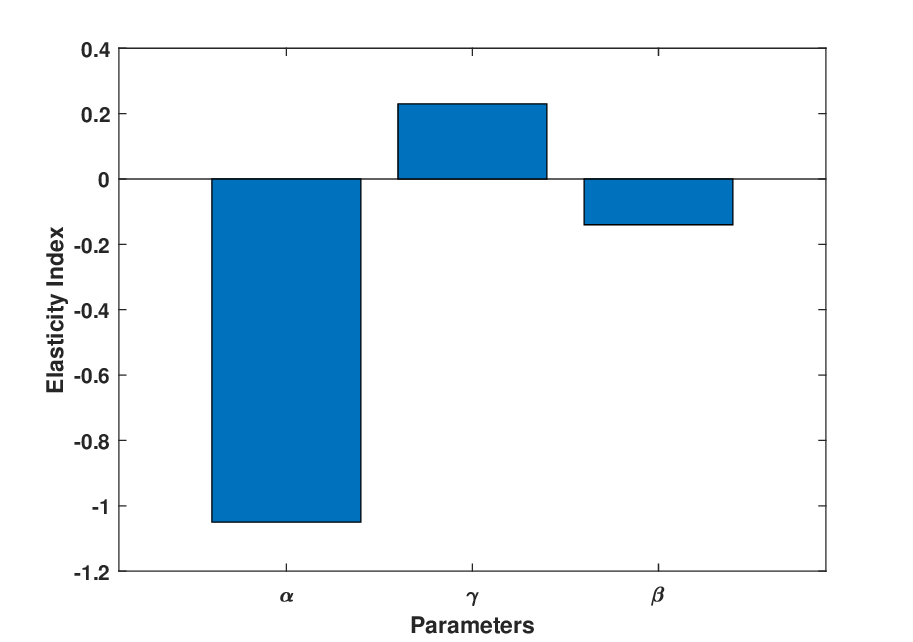};
		\caption{ Elasticities of basic reproduction number with respect to parameters $\alpha$, $\gamma$ and $\beta$.}
		\label{plot_bar: elastricity}
	\end{center}
\end{figure}
\section{Existence of Traveling wave solution}
In order to establish the existence of traveling waves for the system \eqref{Eqn: PDE-simple form rho-1 to rho-5}, we assume that it possesses a solution in the following form $T_1(x,t)=u_1(x+ct),~I_1(x,t)=u_2(x+ct),~D_1(x,t)=u_3(x+ct)$ and $V_1(x,t)=u_4(x+ct)$, where  $u_1,~u_2,~u_3,~u_4$ are the functions of traveling front wave  variable $s=x+ct$, where  $c>0$ denotes the wave speed parameter.  Putting the traveling wave solution in the system \eqref{Eqn: PDE-simple form rho-1 to rho-5}, we obtain
\begin{equation} \label{Eqn: Second order ODE}
	\left.
	\begin{split}
		&cu_1'	= 1-u_1-u_1u_4,\\
		&cu_2'	=u_1u_4-\rho_1 u_2,\\
		&cu_3' =\rho_2 u_2-\rho_3 u_3,\\
		&cu_4'=\mathcal{D}_v u_4''+\rho_4 u_3-\rho_5 u_4.	
	\end{split}
\right\}
\end{equation}
The notation primes $('~\mbox{and}~'')$ signify differentiation with respect to the wave variable $s$. In the context of ecology, it is required that the traveling waves $u_1,~u_2,~u_3$ and $u_4$ should be non-negative and satisfy the following boundary conditions:
\begin{equation} \label{boundary condition}
	\left.
	\begin{split}
		&u_1(-\infty)= 1,~~~~~u_1(\infty)= \frac{\rho _1 \rho _3 \rho _5}{\rho _2 \rho _4}, \\
		&u_2(-\infty)= 0,~~~~~u_2(\infty)= \frac{\rho _2 \rho _4-\rho _1 \rho _3 \rho _5}{\rho _1 \rho _2 \rho _4}, \\	
		&u_3(-\infty)= 0,~~~~~u_3(\infty)= \frac{\rho _2 \rho _4-\rho _1 \rho _3 \rho _5}{\rho _1 \rho _3 \rho _4}, \\
		&u_4(-\infty)= 0,~~~~~u_4(\infty)= \frac{\rho _2 \rho _4-\rho _1 \rho _3 \rho _5}{\rho _1 \rho _3 \rho _5}. \\	
	\end{split}
\right\}
\end{equation}
  Let us denotes $u_5=u_4'$. Subsequently, we derive the following traveling wave equations:
\begin{equation} \label{Eqn: ODE: first order u_1 to u_5}
	\left.
	\begin{split}
		&cu_1'	= 1-u_1-u_1u_4,\\
		&cu_2'	=u_1u_4-\rho_1 u_2,\\
		&cu_3'	=\rho_2u_2-\rho_3 u_3,\\
		&u_4'=u_5,\\
		&\mathcal{D}_vu_5'=c u_5-\rho_4 u_3+\rho_5 u_4.\\ 		 			
	\end{split}
\right\}
\end{equation}
For $R_0>1$, the system of equations \eqref{Eqn: ODE: first order u_1 to u_5} has two steady-states 
$$E_1^*=(1,0,0,0,0), ~\text{and},~
E_2^*=\left(\frac{\rho _1 \rho _3 \rho _5}{\rho _2 \rho _4},~ \frac{\rho _2 \rho _4-\rho _1 \rho _3 \rho _5}{\rho _1 \rho _2 \rho _4},~ \frac{\rho _2 \rho _4-\rho _1 \rho _3 \rho _5}{\rho _1 \rho _3 \rho _4},~ \frac{\rho _2 \rho _4-\rho _1 \rho _3 \rho _5}{\rho _1 \rho _3 \rho _5},0\right)$$
The Jacobian matrix at $E_1^*$ is given by\\
\begin{equation}\label{Matrix: jacobian at E_1^*}
J_{E_1^*}=\left(
\begin{array}{ccccc}
	\cellcolor{blue!20}-\dfrac{1}{c} & 0 & 0 & -\dfrac{1}{c} & 0 \\
	0 &\cellcolor{red!30} -\dfrac{\delta }{c \mu } &\cellcolor{red!30} 0 &\cellcolor{red!30} \dfrac{1}{c} & \cellcolor{red!30}0 \\
	0 &\cellcolor{red!30} \dfrac{a}{c \mu } &\cellcolor{red!30} -\dfrac{R_s}{c \mu } &\cellcolor{red!30} 0 &\cellcolor{red!30} 0 \\
	0 &\cellcolor{red!30} 0\cellcolor{red!30} &\cellcolor{red!30} 0\cellcolor{red!30} &\cellcolor{red!30} 0 &\cellcolor{red!30} 1 \\
	0 &\cellcolor{red!30} 0 &\cellcolor{red!30} -\dfrac{k \alpha  \beta  \lambda }{\mu  D_v} &\cellcolor{red!30} \dfrac{\delta _v}{\mu  D_v} &\cellcolor{red!30} \dfrac{c}{D_v} \\
\end{array}
\right)
\end{equation}
It is clear that $\left(-\dfrac{1}{c}\right)$ is an eigenvalue of the Jacobian  matrix \eqref{Matrix: jacobian at E_1^*}. The other four eigenvalues are the eigenvalues of the sub-matrix \\

\begin{equation}\label{Submatrix 4*4: jacobian at E_1^*}
\left(
\begin{array}{cccc}
	-\dfrac{\delta}{c\mu} & 0 & \dfrac{1}{c}& 0 \\
	\dfrac{a}{c\mu} & -\dfrac{R_s}{c\mu} & 0 & 0 \\
	0 & 0 & 0 & 1 \\
	0 & -\dfrac{ k \alpha  \beta  \lambda}{\mu  D_v} & \dfrac{ \delta _v}{\mu  D_v} & \dfrac{c}{D_v} \\
\end{array}
\right)
\end{equation}
The main aim of this section is to show the existence of traveling wave solution of the system \eqref{Eqn: PDE-simple form rho-1 to rho-5}. 
The original system \eqref{Eqn: PDE-simple form rho-1 to rho-5} admits a traveling wave solution when a heteroclinic orbit connecting the two critical points exist for  the associated ordinary differential equation (ODE) system \eqref{Eqn: ODE: first order u_1 to u_5}. In general, handling matrices (find eigenvalues, eigenvectors, etc.) with dimensions greater than or equal to three is challenging, unless the matrix is  particular. 
To the best of our knowledge, there doesn't exist any hard and fast rule for calculating the eigenvalues of a $4\times 4$ matrix. In our case, the non-zero elements of the  matrix $A$ are function model parameters, and  the matrix is not a special type matrix, hence, the classical or traditional standard approaches for determining the signs of eigenvalue  are difficult to apply. On the other hand, the celebrated $Ger\hat{s}gorin$ theorem provides potential insights to determine  the position  of the eigenvalues of a matrix.  Based on some restrictions on the value of parameters, the existence of  traveling wave is established through the  employment of  well-known $Ger\hat{s}gorin$ theorem in the subsequent section.

\begin{theorem}
	The system \eqref{Eqn: ODE: first order u_1 to u_5} has a traveling wave solution connecting to $E_1^*$ and $E_2^*$ when  the following conditions  hold:
	\begin{enumerate}[(i)]
		\item $\delta>(1+c)\mu$.\label{condition-1}
		\item $R_s-a>c\mu$.\label{condition-2}
		\item $c>D_v\left(1+\dfrac{\delta_v}{\mu D_v}+\dfrac{k\alpha\beta \lambda}{\mu D_v}\right)=c^*.$\label{condition-3}
	\end{enumerate}
\end{theorem}
\begin{proof}
	In order to prove this theorem, we use the Theorem \ref{2012_Horn_Johnson} which is stated in the Section \ref{Preliminaries}. The deleted absolute row sums of $A$  and the centers of the $Ger\hat{s}gorin$ discs are given by: 
	\begin{enumerate}
		\item For first row: $R_1'(A)=\dfrac{1}{c}$, center: $\left(-\dfrac{\delta}{c\mu},0\right)$.
		\item For second row: $R_2'(A)=\dfrac{a}{c\mu}$, center: $\left(-\dfrac{R_s}{c\mu},0\right)$
		\item For third row: $R_3'(A)=1$, center: $(0,0)$
		\item For fourth row: $R_4'(A)=\dfrac{k\alpha\beta\lambda}{\mu D_v}+\dfrac{\delta_v}{\mu D_v}$, center: $\left(\dfrac{c}{D_v},0\right)$
	\end{enumerate}
The  $Ger\hat{s}gorin$ discs are given as 
\begin{align}
	G_1&:=\left\{ z \in \mathcal{C}: \left|z+\dfrac{\delta}{c\mu}\right|\leq R_1'(A)\right\},\\ 
	G_2&:=\left\{ z \in \mathcal{C}: \left|z+\dfrac{R_s}{c\mu}\right|\leq R_2'(A)\right\},\\ 
	G_3&:=\left\{ z \in \mathcal{C}: \left|z-0\right|\leq R_3'(A)\right\},\\ 
	G_4&:=\left\{ z \in \mathcal{C}: \left|z-\dfrac{c}{D_v}\right|\leq R_4'(A)\right\}.
\end{align}
	Since, all the  elements of the  matrix $A$ are real, the disjoint $Ger\hat{s}gorin$ discs gives the real eigenvalues.  For some particular values of parameters, the disjoint $Ger\hat{s}gorin$ discs are shown in Figure \ref{Figure: Four distinct Gersgorin}.  For the values of parameters satisfying conditions \eqref{condition-1}, \eqref{condition-2}, \ref{condition-3}, all  fours dices are disjoint. Based on the proposed conditions (\eqref{condition-1}, \eqref{condition-2}, \eqref{condition-3}),  it is seen that
	\begin{enumerate}
		\item The disc $G_3$ has center at $(0,0)$ and radius $1$. 
		\item The value of $c$ (refer condition \ref{condition-3}) is chosen in such a way that the disc $G_4$ fully lie  in the translated  right  half plane defined by  $\displaystyle \mathcal{P}_{G_4}= \left\{(x,y)\in \mathbb{R}^2:x>1\right\}$, \textit{i.e.}, $G_4$ becomes disjoint from $G_3$. 
		\item Similarly, the conditions \eqref{condition-1} and \eqref{condition-2} are proposed in a manner that 
	 $G_1$ and $G_2$ lies strictly in the plane surfaces defined as $P_{G_1G_2}=\left\{(x,y)\in \mathbb{R}^2:x<-1\right\}$. On the other words, $G_1$, $G_2$ are disjoint from $G_3$ and $G_4$. But  discs $G_1$, $G_2$ may intersect each others.
	\end{enumerate}
	  Now, the right half part of $G_3$ disc lies within the  right half-plane whereas the other part locates in the left half-plane. Therefore, the eigenvalue for the disc $G_3$ is either positive  or negative real number. It can't be zero because the determinant of the matrix $A$ is non-zero. Now, we are in a position to determine the nature of the eigenvalues.
	  \begin{enumerate}
	  	\item The eigenvalue $\lambda_{G3}$ for the disc $G_3$ is either positive or negative real number.
	  	\item The eigenvalue $\lambda_{G4}$ for the disc  $G_4$ is positive real number.
	  	\item The eigenvalues $\lambda_{G1},~\lambda_{G2}$ for the discs $G_1$, $G_2$ are either negative real or  two complex roots with negative real parts.  $\lambda_{G1},~\lambda_{G2}$ can not be purely imaginary due the construction of $G_1$, $G_2$.
	  \end{enumerate}
%
%
	Therefore, based on the condition \eqref{condition-1}, \eqref{condition-2}, \eqref{condition-3},  we always have one positive  eigenvalue. There is no possibility of occurrence complex eigenvalues with positive real parts. 
	 So, system \eqref{Eqn: ODE: first order u_1 to u_5}
	has a traveling wave solution connecting to $E_1^*$ and $E_2^*$. Hence, the system of equations \eqref{Eqn: PDE-simple form rho-1 to rho-5} has traveling wave solution. 
	\begin{figure}
		\includegraphics[height=9cm,width=15cm]{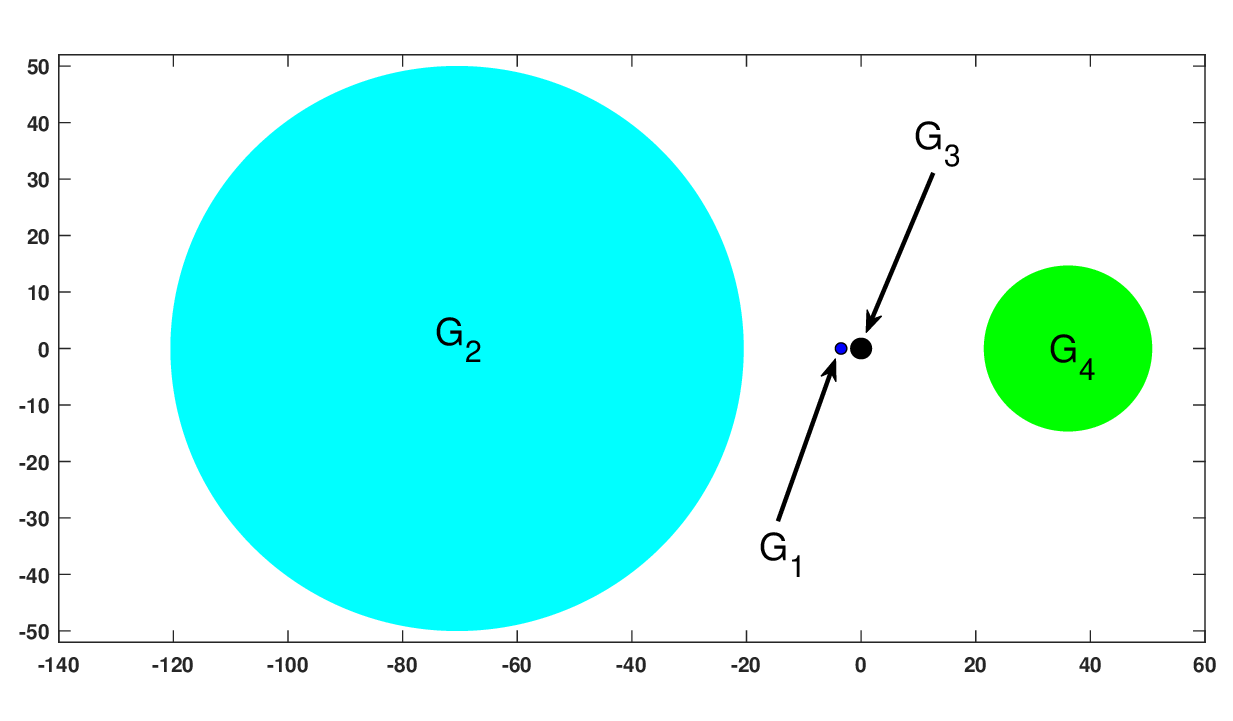}
		\caption{Four distinct $Ger\hat{s}gorin$ dices of the matrix $A$.}
		\label{Figure: Four distinct Gersgorin}
	\end{figure}
\end{proof}
\begin{remark}
	content...
\end{remark}
 \section{Quantitative Analysis}
 The results established in the preceding section demonstrate that system \eqref{Eqn: ODE: first order u_1 to u_5} exhibit a heteroclinic orbit connecting $E_1^*$ and $E_2^*$  when the basic production number $R_0>1$ with the above-mentioned conditions \eqref{condition-1}, \eqref{condition-2}, \eqref{condition-3}. Here,  we investigate how the parameters influence  the infection dynamics. The  parameter values are taken from the Table \ref{Table: Parameter describtion}. For these values of parameters, we calculate $\rho_1=2.81$, $\rho_2=25$, $\rho_3=70.71$, $\rho_4=0.84$, $\rho_5=170$.
 This section also explores the numerical solution of the diffusion model \eqref{Eqn: PDE-simple form rho-1 to rho-5} and provides results for different values of diffusion coefficient.  
 \subsection{Traveling wave solution:}
 For different values of wave speed $c$,  we obtain  non-monotone traveling front profiles for the system \eqref{Eqn: ODE: first order u_1 to u_5} which are shown in Figure \ref{Figure: Traveling wave solution uninfected hepatocytes}. For each values of 
 $c$, a discernible hump is observed in the profiles. The underlying reasons for these occurrences are thoroughly explained in the article of Wang and Wang \cite{2007_wang_wang}. It is further observed that the peak level of the hump  increases while the values of $c$ rises.  
 
 In cases where persistent HBV infection progresses to cirrhosis and primary hepatocellular carcinoma, hepatectomy becomes a necessary intervention for controlling HBV infection in the clinical setting. Generally, determination of appropriate size for hepatectomy poses a challenge in ensuring the effectiveness of the operation. 
 Upon identifying the minimum value of $c$, the size of hepatectomy can be determined accordingly.
 \begin{figure}[h]
 	\includegraphics[height=9cm,width=14cm]{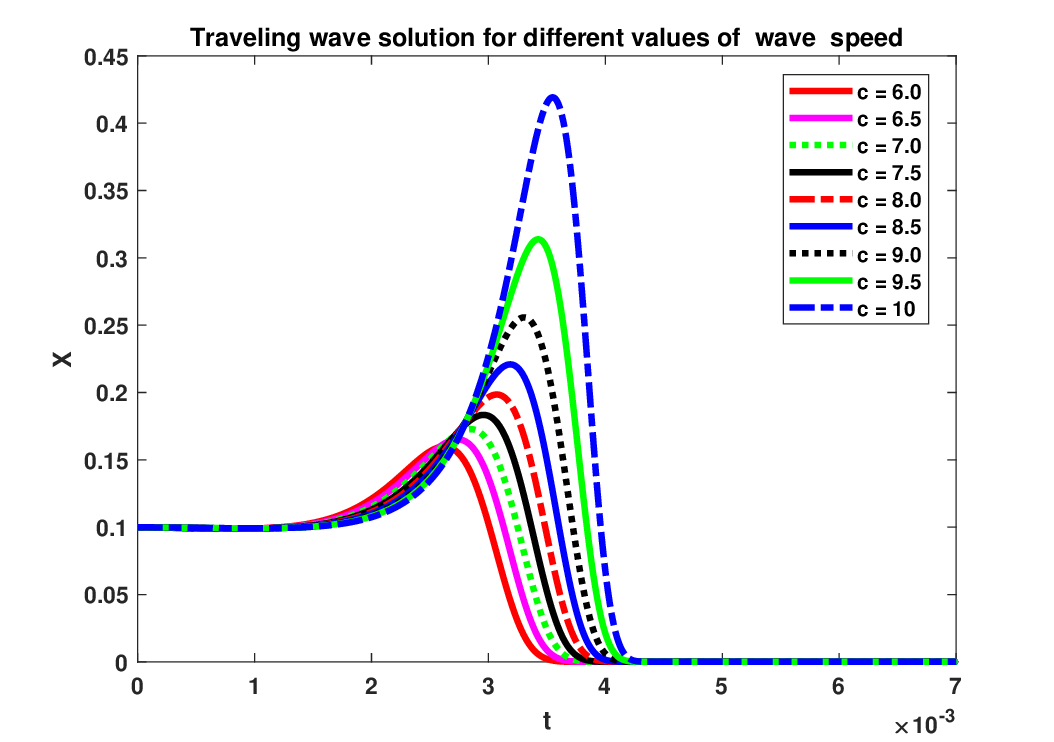}
 	\caption{Traveling wave solution of uninfected hepatocytes for different values of  wave  speed $c$.}
 	\label{Figure: Traveling wave solution uninfected hepatocytes}
 \end{figure}

 \subsection{The effects of diffusion of the viruses:}
 The effects of virus mobility on the HBV infection are investigated through the visualization of the numerical solution. 
 In Figure \ref{Figure: Without diffusion}, the solutions of the proposed model \eqref{Eqn:PDE} are illustrated without diffusion, whereas in Figure \ref{Figure: With diffusion}, the solutions  incorporating  diffusion with diffusion coefficient $d_v=0.2$  are presented.  The values of the parameters are chosen in such a way that $R_0>1$. During the initial stages of infection, the virus concentration is notably elevated at the point of infection. However, the location of the peak level of viruses shifts over time. This shift is particularly significant in non-diffusion systems.  The incorporation of diffusion  enables the virus to spread rapidly throughout the liver. However,  in both cases, the solutions of the system achieve the endemic steady-state $(T_1^*,I_1^*,D_1^*,V_1^*)$ over time.   
 \begin{figure}[h!]
 	\includegraphics[height=10cm,width=14cm]{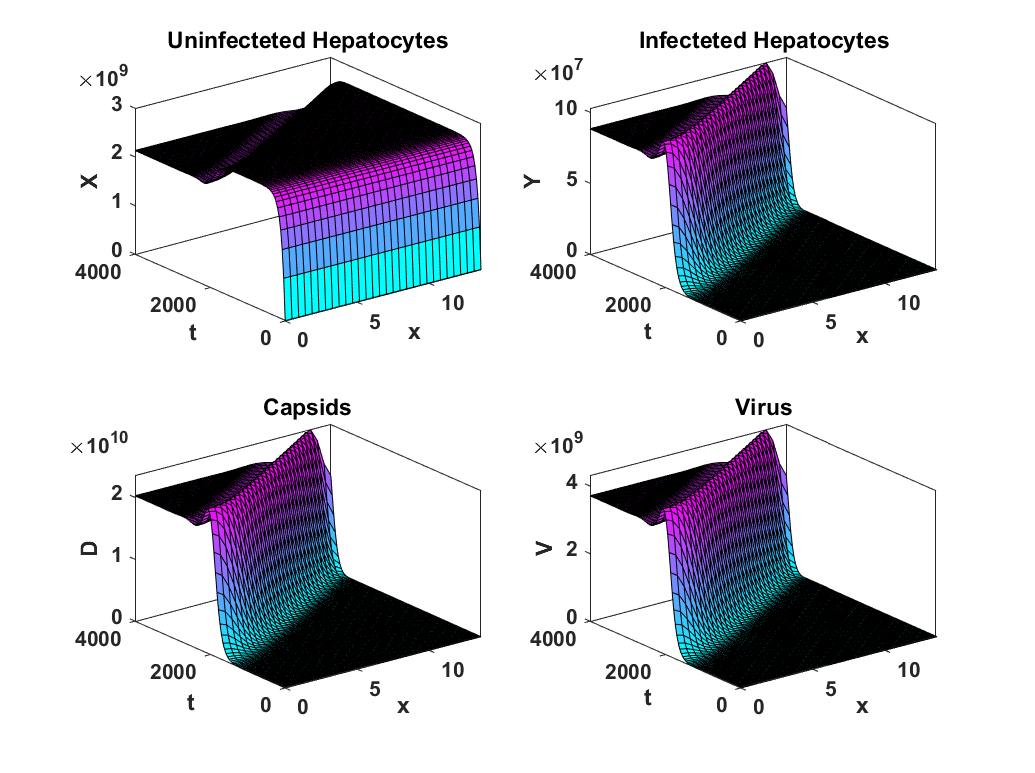}
 	\caption{Solution of system \eqref{Eqn: PDE-simple form rho-1 to rho-5} starting from initial conditions \eqref{eq:initial condition} without diffusion.}
 	\label{Figure: Without diffusion}
 \end{figure}
 \begin{figure}[h!]
	\includegraphics[height=10cm,width=14cm]{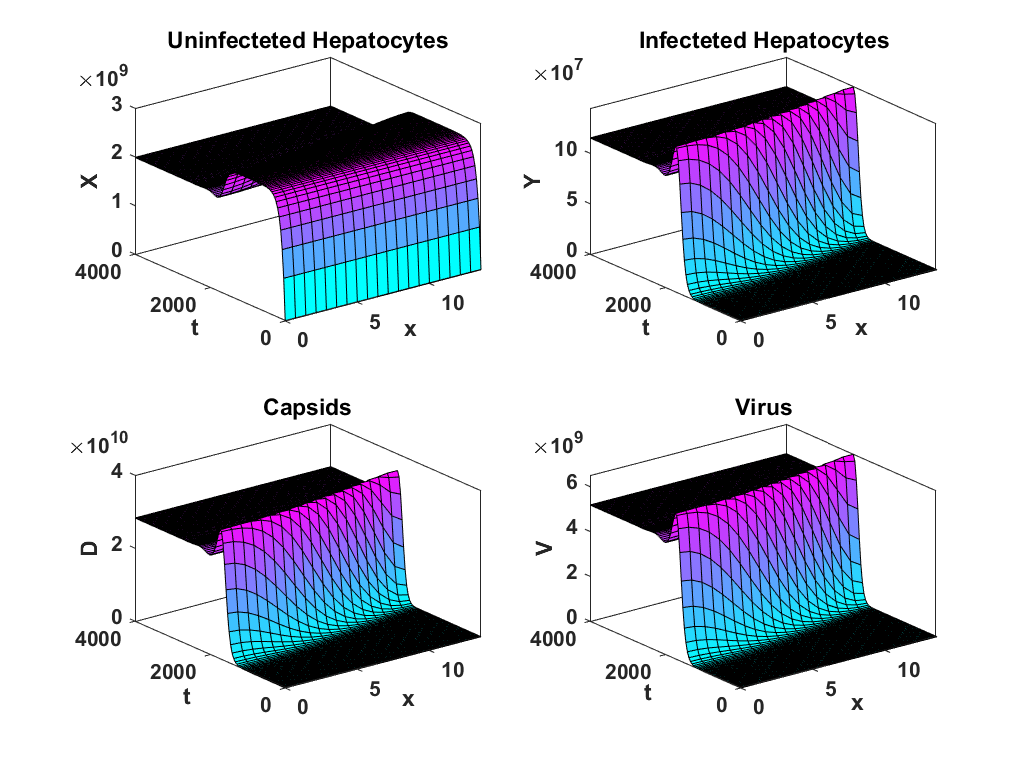}
	\caption{Solution of system \eqref{Eqn: PDE-simple form rho-1 to rho-5} starting from initial conditions \eqref{eq:initial condition} with diffusion. The value of diffusion coefficient is 0.2. }
	\label{Figure: With diffusion}
\end{figure}  
 \section{Conclusions}
 Mathematical models offer valuable insights in understanding the dynamics of infection in vivo  as well as in vitro. In the context of HBV infection, extensive research has been conducted from different aspects to explore the complex interplay of factors which influence virus dynamics within the host. Many factors (cell-to-cell infection, cytolytic and non-cytolytic cure of infected cells, the roles immune cells, etc) can shape the dynamics of HBV infection from various angles. Spatial movement of the virus (diffusion) is one of the key players among them. In this study, including the diffusivity of viruses and recycling of capsids, we extend  non-spatial HBV dynamics model \eqref{ODE_model}, and develop a simple reaction-diffusion model \eqref{Eqn:PDE} for a more realistic representation of the virus dynamics. In order to demonstrate the existence of the traveling wave, we propose an in novo way, probably for the first time, using the celebrated $Ger\hat{s}gorin$ theorem \cite{2012_horn_matrix}. This is the novelty of this work.  We analyze the sensitivity of the basic reproduction number concerning various model parameters, showing  their respective elasticities. Furthermore, our study examines how the spatial diffusivity of viruses affects infection dynamics, and it is observed that the diffusion  facilitates the rapid spread of viruses across the liver.
 
 \subsection*{Acknowledgements}
 The first author also thanks the research facilities received from the Department of Mathematics, Indian
 Institute of Technology Guwahati, India. 
 

\begin{thebibliography}{10}
	
	
	\bibitem{1996_Nowak}
	M. A.  Nowak, S. Bonhoeffer, A. Hill, R. Boehme, H. C.
	Thomas, H. McDade;
	\newblock Viral dynamics in hepatitis B virus infection,
	\newblock {\em Proc. Natl. Acad. Sci. U.S.A.},
	\textbf{93} (1996), 4398--4402.
	
	\bibitem{2006_Murray}
	J. M. Murray, R. H. Purcell, S. F. Wieland;
	\newblock The half-life of hepatitis B virions,
	\newblock {\em Hepatology}, \textbf{44} (2006), 1117--1121.
	
	\bibitem{2007_Ciupe}
	S. M. Ciupe, R. M. Ribeiro, P. W. Nelson, A. S. Perelson;
	\newblock Modeling the mechanisms of acute hepatitis B virus infection,
	\newblock {\em J. Theor. Biol.}, \textbf{247} (2007), 23--35.
	
	\bibitem{2008_Min}
	L. Min, Y. Su,  Y. Kuang;
	\newblock Mathematical analysis of a basic virus infection model with
	application to HBV infection,
	\newblock {\em Rocky Mountain J. Math.}, 38 (2008), 1573--1585.
	
	\bibitem{2015_Manna}
	K. Manna and S. P. Chakrabarty;
	\newblock Chronic hepatitis B infection and HBV DNA-containing capsids:
	Modeling and analysis,
	\newblock {\em Commun Nonlinear Sci Numer Simul},
	\textbf{22} (2015), 383--395.
	
	\bibitem{2018_Danane_mathematical}
	J. Danane, K. Allali;
	\newblock Mathematical analysis and treatment for a delayed hepatitis B viral infection model with the adaptive immune response and DNA-containing capsids,
	\newblock {\em High-throughput}, \textbf{7} (2018), 35.
	
	\bibitem{2018_fatehi_nkcell}
	F. F. Chenar, Y. N. Kyrychko, K. B. Blyuss;
	\newblock Mathematical model of immune response to hepatitis B,
	\newblock {\em J. Theor. Biol.}, \textbf{447} (2018), 98--110.
	
	\bibitem{2021_hews_global}
	S. Hews, S. Eikenberry, J. D. Nagy, T. Phan, Y. Kuang;
	\newblock Global dynamics and implications of an HBV model with proliferating infected hepatocytes,
	\newblock {\em Appl. Sci.}, \textbf{11} (2021), 8176.
	
	\bibitem{2023_Sutradhar_fractional}
	R. Sutradhar, D. C. Dalal;
	\newblock Fractional-order models of hepatitis B virus infection with recycling effects of capsids,
	\newblock {\em Math. Methods Appl. Sci.},
	\textbf{46} (2023), 15599--15625.
	
	\bibitem{2016_jun_nakabayashi}
	J. Nakabayashi;
	\newblock The intracellular dynamics of hepatitis B virus (HBV) replication with reproduced virion “re-cycling”,
	\newblock {\em J. Theor. Biol.}, \textbf{396} (2016), 154--162.
	
	\bibitem{2023_sutradhar_recycling}
	R. Sutradhar, D. C. Dalal;
	\newblock Re-cycling of DNA-containing capsids enhances hepatitis B, \newblock {\em arXiv preprint arXiv: 2309.15665}, (2023).
	
	\bibitem{2004_maini_travelling}
	P. K. Maini, D. L. S. McElwain,  D. Leavesley;
	\newblock Travelling waves in a wound healing assay,
	\newblock {\em Appl. Math. Lett.}, \textbf{17} (2004), 575--580.
	
	\bibitem{2004_matzavinos_travelling}
	A. Matzavinos, M. A. J. Chaplain;
	\newblock Travelling-wave analysis of a model of the immune response to cancer,
	\newblock {\em C. R. Biol.}, \textbf{327} (2004), 995--1008.
	
	\bibitem{2010_gan_travelling}
	Q. Gan, R. Xu, P. Yang, Z. Wu,
	\newblock Travelling waves of a hepatitis B virus infection model with spatial
	diffusion and time delay,
	\newblock {\em IMA J Appl Math}, \textbf{75} (2010), 392--417.
	
	\bibitem{2016_duan_dynamics}
	X. Duan, S. Yuan, K. Wang;
	\newblock Dynamics of a diffusive age-structured HBV model with saturating
	incidence,
	\newblock {\em Math Biosci Eng}, \textbf{13} (2016), 935--968.
	
	\bibitem{2021_issa_diffusion}
	S. Issa, B. M. Tamko, B. Dabol{\'e}, C. B. Tabi, H. P. F. Ekobena;
	\newblock Diffusion effects in nonlinear dynamics of hepatitis B virus,
	\newblock {\em Phys. Scr.}, \textbf{96} (2021), 105217.
	
	\bibitem{2012_horn_matrix}
	R. A. Horn,  C. R. Johnson;
	\newblock { Matrix analysis},
	\newblock Cambridge university press, (2012).
	
	\bibitem{2015_martcheva_introduction}
	M. Martcheva;
	\newblock { An Introduction to Mathematical Epidemiology}, volume~61,
	\newblock Springer, (2015).
	
	\bibitem{2007_wang_wang}
	K. Wang and W. Wang;
	\newblock Propagation of HBV with spatial dependence,
	\newblock {\em Math. Biosci.}, \textbf{210} (2007), 78--95.
	
\end{thebibliography}

\end{document}